\documentclass[11pt]{amsart}
\usepackage{amssymb}

\newtheorem{lemma}{Lemma}[section]

\newtheorem{theorem}[lemma]{Theorem}
\newtheorem{remark}[lemma]{Remark}
\newtheorem{proposition}[lemma]{Proposition}
\newtheorem{corollary}[lemma]{Corollary}

\newtheorem{example}[lemma]{Example}

\begin{document}
\title{Open book structures on semi-algebraic manifolds}

\begin{abstract}
Given a $C^2$ semi-algebraic mapping $F: \mathbb{R}^N \rightarrow \mathbb{R}^p,$ we consider its restriction to $W\hookrightarrow \mathbb{R^{N}}$ an embedded closed semi-algebraic manifold of dimension $n-1\geq p\geq 2$ and introduce sufficient conditions for the existence of a fibration structure (generalized open book structure) induced by the projection $\frac{F}{\Vert F \Vert}:W\setminus F^{-1}(0)\to S^{p-1}$. Moreover, we show that the well known local and global Milnor fibrations, in the real and complex settings, follow as a byproduct by considering $W$ as spheres of small and big radii, respectively. Furthermore, we consider the composition mapping of $F$ with the canonical projection $\pi: \mathbb{R}^{p} \to \mathbb{R}^{p-1}$ and prove that the fibers of $\frac{F}{\Vert F \Vert}$ and $\frac{\pi\circ F}{\Vert \pi\circ F \Vert}$ are homotopy equivalent. We also show several formulae relating the Euler characteristics of the fiber of the projection $\frac{F}{\Vert F \Vert}$ and $W\cap F^{-1}(0).$ Similar formulae are proved for mappings obtained after composition of $F$ with canonical projections.
\end{abstract}
\author{N. Dutertre, R. Ara\'ujo dos Santos, Y. Chen and A. Andrade}

\address{N. Dutertre: Aix-Marseille Universit\'e, CNRS, Centrale Marseille, I2M, UMR 7373,
13453 Marseille, France.}
\email{nicolas.dutertre@univ-amu.fr}

\address{R. N. Ara\'ujo dos Santos, Antonio Andrade and Ying Chen: Departamento de Matem\'{a}tica, Instituto de Ci\^{e}ncias Matem\'{a}ticas e de Computa\c{c}\~{a}o, Universidade de S\~{a}o Paulo - Campus de S\~{a}o Carlos, Caixa Postal 668, 13560-970 S\~ao Carlos, SP, Brazil}
\email{rnonato@icmc.usp.br}
\email{yingchen@icmc.usp.br}
\email{andrade@icmc.usp.br}


\maketitle
\section{Introduction}
For a real analytic mapping germ with isolated critical point at origin $\varphi:(\mathbb{R}^{n},0)\to (\mathbb{R}^{p},0)$, Milnor proved in \cite{Mi} that for all $\epsilon>0$ small enough there exists a smooth locally trivial fibration $S_{\epsilon}^{n-1}\setminus K_{\epsilon}\to S^{p-1},$ where $K_{\epsilon}:=\varphi^{-1}(0)\cap S_{\epsilon}^{n-1},$ and pointed out that the map $\frac{\varphi}{\|\varphi\|}$ may fail to be the projection of such fibration since the singular locus $\Sigma (\frac{\varphi}{\|\varphi\|})$ may have a curve inside $B_{\epsilon} \setminus \varphi^{-1}(0)$ which accumulates at the origin. See \cite{Mi}, page 100, for details.

\vspace{0.2cm}

For germs with non isolated singularity, the existence of a fibration where the canonical projection $\frac{\varphi}{\|\varphi\|}$ extends to spheres was studied in \cite{AT, ACT1, CSS} and called {\it (Singular) Higher open book structures on spheres}.

\vspace{0.2cm}

In the global setting, N\'emethi and Zaharia in \cite{NZ} used Milnor's method to show that if a polynomial function $f: \mathbb{C}^{n}\to \mathbb{C}$ satisfies the semitame condition, then the projection $\displaystyle{\frac{f}{\|f\|}:S_{R}^{2n-1}\setminus K_{R}\to S^1}$ is a locally trivial fibration for all $R$ big enough. In \cite{ACT2} the authors generalized the results of N\'emethi and Zaharia introducing a kind of Milnor's conditions at infinity for a class of mappings which contains complex semitame functions, several kind of real polynomials maps and meromorphic functions. In fact, under such conditions they proved the existence of a ``Singular open book structure at infinity". It means that, given a polynomial map $F:\mathbb{R}^{n}\to \mathbb{R}^{p},$ $n>p\geq 2,$ satisfying Conditions (a) and (b) (see Section $1$ for details), there exists $R_{0}\gg 1$ such that, for all $R\geq R_{0},$ the projection map
$$\frac{F}{\|F\|}:S_{R}^{n-1}\setminus K_{R}\to S^{p-1}$$
is a smooth locally trivial fibration, where $K_{R}=F^{-1}(0)\cap S_{R}^{n-1}$ is called the link at infinity.

Some examples of such open book structures are given by the mixed polynomial functions, which are real polynomial mappings from $\mathbb{C}^{n}$ to $\mathbb{C}$ in the complex variables and their conjugates (see \cite{Oka} in local setting and \cite{Ch} in the global setting).

\vspace{0.2cm}

In this paper we prove a general fibration theorem (Theorem \ref{t:fibration}) for a more general class of maps and sets. For this we consider a $C^2$ semi-algebraic map $F=(f_1,\ldots,f_p): \mathbb{R}^N \rightarrow \mathbb{R}^p$ and $W\hookrightarrow \mathbb{R}^N$ an embedded compact semi-algebraic manifold without boundary of dimension $n-1\geq p$. We introduce sufficient conditions that ensure the existence of an open book structure on $W$ induced by the projection $\frac{F}{\Vert F \Vert}$. As a consequence we extend both fibration structures: the local and global ones in the real and complex settings.

In another direction, we investigate the composition mapping of $F$ with the canonical projection $\pi: \mathbb{R}^{p} \to \mathbb{R}^{p-1}$. In the local setting this problem was approached in \cite{DA} for the so-called Milnor's tube fibration and it was shown that the canonical projection $\pi$ does not change the homotopy type of the fiber. In the present paper we prove that an analogous result holds true for the general fibration $\displaystyle{\frac{F}{\|F \|}: W\setminus F^{-1}(0) \rightarrow  S^{p-1}}.$ Moreover, it is shown that the Euler characteristics of the fiber of the projection $\frac{F}{\Vert F \Vert}$ and $W\cap F^{-1}(0)$ are closely related, as a consequence, it follows that the Euler characteristic of the fiber is a natural obstruction for the sets $W\cap F^{-1}(0)$ and $W\cap (\pi\circ F)^{-1}(0)$ to be homotopy equivalent. See Proposition \ref{p:Euler}.

\vspace{0.3cm}

The structure of the paper is as follows. Section $2$ is devoted to establish a general fibration theorem for $C^2$ semi-algebraic mappings. In Section $3$, we consider the composition of the mapping $F$ with the canonical projection and show that the respective fibers are homotopy equivalent. In Section $4,$ we relate the Euler characteristics of the fibers and the relative links. In the last section, we explain how to relate our results to the higher open book structure for polynomial mappings in the local and global settings and calculate some examples.

\section{Fibration structure} \label{section 2}

\subsection{Fibration theorem}

Let $W\hookrightarrow \mathbb{R}^N$ be a smooth, compact semi-algebraic $(n-1)-$dimensional submanifold embedded in $\mathbb{R}^N.$ Let $F=(f_1,\ldots,f_p): \mathbb{R}^N \rightarrow \mathbb{R}^p$, $2  \le p \le n-1$, be a $C^2$ semi-algebraic map, $V(F) = F^{-1}(0)$ and $V_W(F)=V(F)\cap W$. We denote by $\bar{ F}: \mathbb{R}^N \setminus V(F) \rightarrow S^{p-1}$ the projection $\frac{F}{\Vert F \Vert}$ where $\|F\|=\sqrt{f_{1}^{2}+\cdots +f_{p}^{2}}.$ We also denote by:

\begin{itemize}
  \item $\Sigma_F$ the set of critical points of $F,$
	\vspace{0.2cm}
	
  \item $\Sigma_{\bar{F}}$ the set of critical points of $\bar{F},$
	\vspace{0.2cm}
	
  \item $\Sigma_F^W$ the set of critical points of $F_{\vert W},$
	\vspace{0.2cm}
	
  \item $\Sigma_{\bar{F}}^W$ the set of critical points of $\bar{F}_{\vert W}.$
\end{itemize}

\vspace{0.2cm}

It follows by definition (see also Remark \ref{re:generators} and Lemma \ref{matrix} below) that $\Sigma_{\bar{F}}$ is a subset of $\Sigma_F $ and $\Sigma_{\bar{F}}^W$ is a subset of $\Sigma_F^W$.

\vspace{0.2cm}

We are interested in studying the existence of the locally trivial fibration
\begin{equation}\label{eq:openbook}
\bar{F}:  W\setminus V_W(F)   \rightarrow  S^{p-1}.
\end{equation}
We call such fibration the {\it fibration structure induced by
$\bar{F}$}.

\vspace{0.2cm}

We should point out that Conditions (a) and (b) and the proof of Theorem \ref{t:fibration} below follow closely the one stated in \cite{AT, ACT1, ACT2}, which in turn were inspired by Massey's conditions introduced in \cite{DM}. The idea of our result is to show how to adapt them in a more general setting. For this, consider the following conditions:

\vspace{0.2cm}

\begin{itemize}
\item[(a)] $\overline{\Sigma_F^W \setminus V_W(F)} \cap V_W(F)= \emptyset$,
\item[(b)] $\Sigma_{\bar{F}}^W = \emptyset$.
\end{itemize}

\vspace{0.2cm}


Without lost of generality, we may assume $W$ connected.

\begin{theorem}\label{t:fibration} Under Condition (a), the following statements are equivalent:

\vspace{0.2cm}

 \begin{itemize}
\item[(i)] $\bar{F}:  W\setminus V_W(F) \rightarrow  S^{p-1}$ is a $C^{2}$ locally trivial fibration,

\vspace{0.3cm}

\item[(ii)] Condition (b) holds.
\end{itemize}
\end{theorem}

\begin{proof}
\noindent
Firstly, see that Condition (b) holds if and only if:
$$\bar{F}:  W\setminus V_W(F)\rightarrow  S^{p-1}$$
is a submersion.

By the Curve Selection Lemma, Condition $(a)$ means that:
there exists $\delta>0$ small enough and a closed disc $D_\delta$ centered at origin such that
$$ F_{|}: F^{-1}(D_\delta-\{0\})\cap W \rightarrow D_\delta-\{0\}$$
is a $C^2$ surjective proper submersion.

\vspace{0.2cm}

Let us prove the implication (i) $\Rightarrow$ (ii). If we suppose that $\bar{F}: W\setminus V_W(F) \rightarrow  S^{p-1}$ is a locally trivial fibration, then it is a submersion and so Condition (b) follows.

\vspace{0.2cm}

Let us prove the converse. We can assume that $V_W(F)$ is not empty: in fact, in the case it is empty, Condition (a) is always satisfied and by Condition (b), we have that $\bar{F}: W \rightarrow S^{p-1}$ is a $C^{2}$ proper submersion and so onto, since $S^{p-1}$ is connected. Hence (i) follows by Ehresmann's theorem.

In the case $V_W(F)$ is not empty, the projection $\bar{F}: W\setminus V_W(F)\rightarrow S^{p-1}$ is not proper and we can not use Ehresmann's theorem. Nevertheless, the mapping $ F_{|}: F^{-1}(D_\delta-\{0\})\cap W \rightarrow D_\delta-\{0\}$ is a
surjective proper submersion. So, by Ehresmann's theorem, the mapping
\begin{equation}\label{eq:fib}
F_{|}: F^{-1}(D_\delta-\{0\})\cap W \rightarrow D_\delta-\{0\}
\end{equation}
is a locally trivial fibration. Now, we can compose it with the radial projection $\pi_{1}: D_\delta-\{0\} \rightarrow S^{p-1}, $ $\pi_{1} (y)=\frac{y}{\|y\|},$ to get that
\begin{equation}\label{eq:fib1}
\bar{F}: F^{-1}(D_\delta-\{0\})\cap W \rightarrow S^{p-1}
\end{equation}
is a locally trivial fibration.

Fibration (\ref{eq:fib}) yields that the map

\begin{equation}\label{eq:fib 3}
\bar{F}: F^{-1}(S_\delta)\cap W \rightarrow S^{p-1}
\end{equation}
is also a locally trivial fibration and therefore surjective.

This implies that
\begin{equation}\label{eq:fib2}
\bar{F}:  W\setminus F^{-1}(\mathring{D}_\delta)\rightarrow  S^{p-1}
\end{equation}
is surjective and proper, where $\mathring{D}_\delta$ denotes the open disk. So, by using Condition (b), the mapping $\bar{F}$ is a $C^{2}$ submersion. Therefore, we have a locally trivial fibration by Ehresmann's theorem.

Now we can glue fibrations \eqref{eq:fib1} and \eqref{eq:fib2} along the common boundary $F^{-1}(S_\delta)\cap W,$ using fibration (\ref{eq:fib 3}), to get the locally trivial fibration (1). \end{proof}

\begin{remark}

\begin{enumerate}

\item  The main argument in the proof is based on Ehresmann's theorem. Therefore, it is possible to generalize Conditions (a) and (b) above and Theorem \ref{t:fibration} for $C^2$ mappings and manifolds not necessarily in the semi-algebraic category.

\item It is possible to change the submanifold $W$ by a compact semi-algebraic set with non-empty singular locus $\Sigma(W)$ if one just asks that $\Sigma(W)\subseteq V_W(F).$ Alternatively, one can change the compactness condition on $W$ by asking $F$ to be proper on $W$. In both cases, the statement of Theorem \ref{t:fibration} and its proof work in the same way.

\end{enumerate}
\end{remark}

\section{Canonical projections and fibers}

In this section we study the behaviour of Conditions (a) and (b) after considering composition of $F$ with canonical projections. The idea was inspired by the studies developed in \cite[Theorem 6.2]{DA}, where it was shown that for real analytic mapping germs with isolated critical values, under the so-called Milnor's conditions (a) and (b) introduced by D. Massey in \cite{DM}, the canonical projection preserves the homotopy type of the Milnor fiber.

From now on, we follow the notations used in section \ref{section 2}. The following statement is an important remark which provides a way to understand the generators of the normal space to the fibers of $\bar{F}: \mathbb{R}^N \setminus V(F) \rightarrow S^{p-1}.$
\begin{remark}\label{re:generators}
Let $\pi_i : \mathbb{R}^p \rightarrow \mathbb{R}^{p-1}$ be the projection $$(y_1,\ldots,y_p) \mapsto (y_1, \ldots, \hat{y_i}, \ldots, y_p).$$
The restriction of $\pi_{i}\circ\bar{F}$ on the open set $\{f_i \not= 0 \}$ is given by
$$\begin{array}{ccccc}
\pi_ i \circ \bar{F} & : & \mathbb{R}^N \setminus V(F)  & \rightarrow  & B^{p-1} \subset \mathbb{R}^{p-1} \cr 
   &   &  x  & \mapsto & (\bar{f_1}(x), \ldots, \widehat{\bar{f_i}(x)}, \ldots, \bar{f_p}(x) ), \cr
\end{array}$$
where $\bar{f_i}$ denotes the function $\frac{f_i}{\Vert F \Vert}$ for $i=1,\ldots,p$. Moreover we have that $\Sigma_{\bar{F}}\cap \{f_{i} \not= 0\}$ is the set of points where
$${\rm rank } \left[ \begin{array}{c}
\nabla \bar{f_1} \cr
\vdots \cr
\nabla \bar{f}_{i-1} \cr
\nabla \bar{f}_{i+1} \cr
\vdots \cr
\nabla \bar{f_{p}} \cr
\end{array}
\right] < p-1.$$
\end{remark}
Let us see below a way to make these generators explicit. As in \cite{AT, ACT2}, for each $1\leq i,j \leq p,$ let us denote by
$$\omega_{i,j}(x)=f_{i}(x)\nabla f_{j}(x)-f_{j}(x)\nabla f_{i}(x).$$
It follows that $\omega_{i,i}(x) = 0$ and $\omega_{j,i}(x)=-\omega_{i,j}(x).$ Calculations show that for each $1\leq j \leq p,$


\begin{equation}\label{eq:projection}
\nabla(\frac{f_{j}}{\|F\|})(x)=\frac{1}{\|F\|^{3}}\sum_{1\leq k\neq j}^{p} f_{k}(x)\omega_{k,j}(x).\hspace{1cm}
\end{equation}

It means that, if $x$ is not a singular point of $\bar{F}$, then the set of vectors $\{\omega_{i,j}(x)\}_{1\leq i<j\leq p}$ span the normal space in $\mathbb{R}^{n}\setminus V(F)$ to the fiber $\bar{F}^{-1}(\bar{F}(x))$ at $x.$ It is also easy to check that the following circular relation holds.

\begin{lemma}\label{lem:eqmilnor}
 For all $1\leq i<j<k\leq p$ we have:

$$\displaystyle{f_{i}\omega_{j,k}+f_{k}\omega_{i,j}+f_{j}\omega_{k,i}=0}.\hspace{1cm}(2)$$

\end{lemma}

\proof It follows by definition. \endproof

Since this relation contains the main information concerning the normal space to the fiber, we call it {\bf Milnor's equality}.

\begin{lemma} \label{matrix}

Given $F$ as above, for all $x$ in the open set $\{f_{1}(x)\neq 0\}$, we have the following matrix equation:

$$ \|F(x)\|^{3p}\cdot\begin{pmatrix}
\displaystyle{\nabla (\frac{f_{1}}{\|F\|})(x)} \\
\displaystyle{\nabla (\frac{f_{2}}{\|F\|})(x)} \\
\ldots \\
\displaystyle{\nabla (\frac{f_{p}}{\|F\|})(x)}
\end{pmatrix} _{p\times n} =$$

$$\hspace{-1cm} \begin{pmatrix}
-f_{2} & -f_{3} & \ldots & -f_{p-1} & -f_{p}\\
\frac{f_{1}^{2}+f_{3}^{2}+\cdots +f_{p}^{2}}{f_{1}} & -\frac{f_{2}f_{3}}{f_{1}} & \ldots & -\frac{f_{2}f_{p-1}}{f_{1}} & -\frac{f_{2}f_{p}}{f_{1}}\\
-\frac{f_{2}f_{3}}{f_{1}} & \frac{f_{1}^{2}+f_{2}^{2}+\cdots +f_{p-1}^{2}+f_{p}^{2}}{f_{1}} & \ldots & -\frac{f_{3}f_{p-1}}{f_{1}} & -\frac{f_{3}f_{p}}{f_{1}}\\
\cdots & \cdots & \cdots & \ldots & \cdots & \\
-\frac{f_{2}f_{p}}{f_{1}} & -\frac{f_{3}f_{p}}{f_{1}} & \ldots & -\frac{f_{p-1}f_{p}}{f_{1}} & \frac{f_{1}^{2}+f_{2}^{2}+\cdots +f_{p-1}^{2}}{f_{1}}
\end{pmatrix}_{p\times (p-1)}
\begin{pmatrix}
\omega_{1,2}(x) \\
\omega_{1,3}(x)\\
\ldots \\
\omega_{1,p}(x)
\end{pmatrix}_{(p-1)\times n},
$$
where the $p\times (p-1)$ matrix has maximal rank.

\end{lemma}
\proof The matrix equation follows from relation \eqref{eq:projection} and Lemma \ref{lem:eqmilnor}.

\vspace{0.3cm}

Now consider the $(p-1)\times (p-1)$ submatrix, let us say $A(x)$, built by removing the first line in the $p\times (p-1)$ matrix above:

$$A(x)=\begin{pmatrix}
\frac{f_{1}^{2}+f_{3}^{2}+\cdots +f_{p}^{2}}{f_{1}} & -\frac{f_{2}f_{3}}{f_{1}} & -\frac{f_{2}f_{4}}{f_{1}} & \ldots & -\frac{f_{2}f_{p}}{f_{1}}\\
-\frac{f_{2}f_{3}}{f_{1}} & \frac{f_{1}^{2}+f_{2}^{2}+\cdots +f_{p-1}^{2}+f_{p}^{2}}{f_{1}} & -\frac{f_{3}f_{4}}{f_{1}} & \ldots & -\frac{f_{3}f_{p}}{f_{1}}\\
\cdots & \cdots & \cdots & \ldots & \cdots & \\
-\frac{f_{2}f_{p}}{f_{1}} & -\frac{f_{3}f_{p}}{f_{1}} & -\frac{f_{4}f_{p}}{f_{1}} & \ldots & \frac{f_{1}^{2}+f_{2}^{2}+\cdots +f_{p-1}^{2}}{f_{1}}
\end{pmatrix}.$$

Since it is a symmetric matrix it is diagonalizable, then the geometric and algebraic multiplicity are the same. Now, by linear algebra calculations one can prove that $\displaystyle{\frac{\|F(x)\|^{2}}{f_{1}(x)}}$ is an eigenvalue with multiplicity $(p-2)$ and $f_{1}(x)$ is an eigenvalue with multiplicity one. So, the determinant of $A(x)$ is equal to $\dfrac{\|F (x)\|^{2(p-2)}}{f_{1}^{p-3}(x)}$ and never vanishes on the open set $\{f_{1}(x)\neq 0\}.$
 Therefore, the linear operator $A(x)$ is injective and the ranks of the $p \times n$ matrix on the left side and the $(p-1)\times n$ matrix on the right side are the same. \endproof

\begin{remark}

The above lemma still holds true with the same proof, if one considers for each $i=2,\cdots, p,$ the point $x$ in the open set $f_{i}(x)\neq 0.$ In this case, the $i$-th line of the $p\times (p-1)$ matrix is given by 

$$
\left\{\begin{array}{c}
a_{ik} = -f_{k}, \mbox{ if }  k<i,\\\\
a_{ik} = -f_{k+1}, \mbox{ if }  k\geq i.\\\\
\end{array}\right.
$$

Now, removing this line from the matrix, one obtains a $(p-1)\times (p-1)$ symmetric matrix such that the elements in the diagonal are given by 

$$
\left\{\begin{array}{c}
a_{jj} = \displaystyle \sum_{1\leq k \neq j}^{p} \frac{f^{2}_{k}}{f_{i}}, \mbox{ if } j<i, \\\\
a_{jj} = \displaystyle \sum_{1\leq k \neq j+1}^{p} \frac{f^{2}_{k}}{f_{i}}, \mbox{ if } j\geq i,\\\\
\end{array}\right.
$$

and the elements above the diagonal are given by

$$
\left\{\begin{array}{c}
a_{lk} = -\frac{f_{l}f_{k}}{f_{i}}, \mbox{ if } k<i, \\\\
a_{lk} = -\frac{f_{l}f_{k+1}}{f_{i}}, \mbox{ if } k\geq i. \\\\
\end{array}\right.
$$

\noindent The $(p-1)\times n$ matrix in this case takes the form

$$\begin{pmatrix}
\omega_{i,1}\\
\omega_{i,2}\\
\vdots \\
\omega_{i,i-1}\\
\omega_{i,i+1}\\
\vdots \\
\omega_{i,p}
\end{pmatrix}_{(p-1) \times n}.
$$
\end{remark}

\subsection{\textbf{Case $p\ge 3$}}
We assume that $p\ge 3$ and consider the canonical projection $\pi: \mathbb{R}^{p} \to \mathbb{R}^{p-1}$, $\pi(y_1,\ldots,y_p)=(y_1, \ldots, {y_i}, \ldots, y_{p-1})$. We denote the composition mapping $\pi\circ F$ by $G=(f_1,\ldots,f_{p-1})$ and as before $\bar{G}=\frac{G}{\|G\|}.$

\begin{lemma}\label{lem:fib}
If $F$ satisfies Conditions (a) and (b), then $G=(f_1,\ldots,f_{p-1})$ also satisfies Conditions (a) and (b).
\end{lemma}
\begin{proof}
We prove first that Condition (b) holds for $G$. Assume that there exists $x\in\Sigma^W_{\bar{G}}$. Then, there exists $ i \in \{1,\ldots,p-1\}$ such that $f_i(x) \not= 0$.
 Since $W\subset \mathbb{R}^{N}$ is a smooth semi-algebraic manifold of dimension $n-1$, locally around $x$ we can assume that $W=\{h_1(x)=0,\ldots,h_{N-n+1}(x)=0\}$ where
$h_1,\ldots,h_{N-n+1}$ are smooth semi-algebraic functions and $0$ is a regular value of the mapping $(h_1,\ldots,h_{N-n+1}):\mathbb{R}^{N} \to \mathbb{R}^{N-n+1}$. By Remark \ref{re:generators}, at the point $x$ we have
$${\rm rank } \left[ \begin{array}{c}
\nabla h_1(x) \cr
\vdots \cr
\nabla h_{N-n+1}(x) \cr
\nabla \bar{f_1}(x) \cr
\vdots \cr
\nabla \bar{f}_{i-1}(x) \cr
\nabla \bar{f}_{i+1}(x) \cr
\vdots \cr
\nabla \bar{f}_{p-1}(x) \cr
\end{array}
\right] < N-n+1+p-2,$$
and by adding in the last line the vector $\nabla \bar{f_{p}}(x),$ we get
$${\rm rank } \left[ \begin{array}{c}
\nabla h_1(x) \cr
\vdots \cr
\nabla h_{N-n+1}(x) \cr
\nabla \bar{f_1}(x) \cr
\vdots \cr
\nabla \bar{f}_{i-1}(x) \cr
\nabla \bar{f}_{i+1}(x) \cr
\vdots \cr
\nabla \bar{f_{p}}(x) \cr
\end{array}
\right] < N-n+1+p-1.$$
Hence, $x \in \Sigma^W_{\bar{F}}$ which is in contradiction since $\Sigma_{\bar{F}}^W = \emptyset$.
So, $G$ satisfies Condition (b).

\vspace{0.3cm}

Next let us prove that Condition (a) also holds for $G,$ i.e., we have to show that
$$\overline{\Sigma^W_{G} \setminus V_W(G)} \cap V_W (G) = \emptyset.$$
Observe that
$$\Sigma^W_G \subset \Sigma^W_F \hbox{ and } V_W (F) \subset V_W (G),$$
then from the inclusions
$$ \Sigma^W_G \setminus V_W (G) \subset  \Sigma^W_F\setminus V_W (F),$$
and
$$\overline{ \Sigma^W_G \setminus V_W (G)} \subset  \overline{\Sigma^W_F\setminus V_W (F)},$$
the inclusion 
\begin{equation}\label{eq:inclusion1}
\overline{ \Sigma^W_G \setminus V_W (G)} \cap V_W(F) \subset \overline{\Sigma^W_F\setminus V_W (F)} \cap V_W (F),
\end{equation}
follows. 

Since $F$ satisfies Condition (a), we get that $\overline{\Sigma^W_G \setminus V_W (G)}\cap V_W(F)=\emptyset$.

On $V_W(G)\setminus V_W(f_p)$, we have $f_p \not= 0$. By Remark \ref{re:generators}, $x\in\Sigma^W_{\bar{F}}\setminus V_W(f_p)$ if and only if
$${\rm rank } \left[ \begin{array}{c}
\nabla h_1(x) \cr
\vdots \cr
\nabla h_{N-n+1}(x) \cr
\nabla \bar{f}_1(x) \cr
\vdots \cr
\nabla \bar{f}_{p-1}(x) \cr
\end{array}
\right] < N-n+1+p-1.$$

On the other hand, on $V_W(G)$, for $i=1,\ldots,p-1$, we have
$$\nabla \bar{f}_i = \frac{1}{\Vert F \Vert^2} \left( \Vert F \Vert \nabla f_i - f_i \nabla \Vert F \Vert \right)= \frac{1}{\Vert F \Vert} \nabla f_i.$$
Hence $x\in\Sigma^W_{\bar{F}} \cap \{f_p \not= 0 \} $ if and only if $x \in \Sigma^W_G \cap \{f_p \not= 0 \}$.

Therefore we get
\begin{equation}\label{eq:equality1}
\Sigma^W_{\bar{F}} \cap \left[ V_W(G) \setminus V_W (f_p) \right] = \Sigma^W_{G} \cap \left[ V_W(G) \setminus V_W (f_p) \right].
\end{equation}
We have
$$\overline{\Sigma^W_{G}\setminus V_W(G)}\cap V_W (G) = \overline{\Sigma^W_{G} \setminus V_W(G)} \cap \left[(V_W(G) \setminus V_W(f_p))\cup V_W(F) \right]=$$
$$ \overline{\Sigma^W_{G} \setminus V_W(G)} \cap V_W(F) \bigcup \overline{\Sigma^W_{G} \setminus V_W(G)} \cap \left[ V_W(G) \setminus V_W(f_p) \right].$$
Since $\Sigma^W_{G}$ is closed and $\overline{\Sigma^W_G \setminus V_W (G)}\cap V_W(F)=\emptyset$, we conclude that
$$\overline{\Sigma^W_{G} \setminus V_W(G)} \cap V_W(G) \subset \Sigma^W_{G} \cap \left[ V_W(G) \setminus V_W(f_p) \right].$$
It follows from Equality \eqref{eq:equality1} that:
$$\overline{\Sigma^W_{G} \setminus V_W(G)} \cap V_W(G) \subset \Sigma^W_{\bar{F}} \cap \left[ V_W(G) \setminus V_W(f_p) \right].$$
By Condition (b), the right-hand side of the above inclusion is empty, so
$$\overline{\Sigma^W_G \setminus V_W(G)} \cap V_W (G) = \emptyset.$$
 Therefore Condition (a) holds for $G$.
\end{proof}

According to Lemma \ref{lem:fib} and Theorem \ref{t:fibration} the projection $\bar{G}: W \setminus V_W(G) \rightarrow S^{p-2}$ is also a locally trivial fibration.

\vspace{0.3cm}

Denote by $M_F$ and $M_G$ the fibers of $\bar{F}: W \setminus V_W(F) \rightarrow S^{p-1}$ and $\bar{G}: W \setminus V_W(G) \rightarrow S^{p-2},$ respectively. Now we can relate them as follows.

\begin{theorem}\label{t:homotopy}
The fiber $M_G$ is homotopy equivalent to the fiber $M_F$.
\end{theorem}

Before starting the proof we need to state some important remarks and results. Without lost of generality, we can consider the fiber

$$M_F = \bar{F}^{-1} (1,0,\ldots,0) \cap (W  \setminus V_W(F)).$$

So,
$$ M_F = \left\{ f_1 >0, f_2=\ldots=f_p=0  \right\} \cap (W \setminus \{f_1=f_2=\ldots=f_p =0 \}) =$$ $$\left\{ f_1 >0, f_2=\ldots=f_p=0  \right\} \cap (W \setminus \{f_1 =0 \}).$$
Hence,
$$M_F = \left\{ f_1 >0, f_2=\ldots=f_p=0 \right\} \cap W.$$
Similarly, $$M_G =  \left\{ f_1 >0, f_2=\ldots=f_{p-1}=0 \right\} \cap W.$$
Note that $\overline{M_F} =M_F \cup V_W(F)$ and that $V_W(F) = \overline{M_F} \cap \{f_1 = 0 \}$.

\begin{lemma}\label{lem:carpeting}
The functions ${f_1}_{\vert M_{F}}$ and ${f_1}_{\vert M_{G}}$ are carpeting functions (see \cite{TH}) for $M_F$ and $M_G$.
\end{lemma}
\proof Under our hypothesis, the function $f_1$ is strictly positive on $M_F$. Let us prove that there exists $\delta >0$ such that $f_1$ is a submersion on $M_F \cap \{0 < f_1 \le \delta \}$.
If it is not the case, we can find a sequence $(\delta_n)_{n \in \mathbb{N}}$ such that $\delta_n >0$, $\delta_n \rightarrow 0$ and $\{ f_1= \delta_n\}$ does not intersect $V_W(\eta)$ transversally, where $\eta$ denotes the mapping $(f_2,\ldots,f_p)$.
This implies that there exists $y_n$ in $\{f_1= \delta_n \}$ such that $y_n \in \Sigma^W_F \setminus V_W(F)$.
Since $W$ is compact, we have $\overline{\Sigma^W_F \setminus V_W(F)} \cap V_W(F) \not= \emptyset$ which contradicts Condition (a) for $F$.
Hence $f_1$ is a carpeting function for $M_F$. The same proof still works for $M_G$, and the lemma follows.
\endproof

\begin{corollary}\label{HomEqu}
The fiber $M_F$ (resp. $M_G$) is homotopy equivalent to $\{f_1 \ge \delta, f_2=\ldots=f_p=0 \} \cap W$ (resp. $\{f_1 \ge \delta, f_2=\ldots=f_{p-1}=0 \} \cap W$).
\end{corollary}
\proof Let us consider $M_F.$ This fiber is diffeomorphic to $\{f_1> \delta, f_2=\ldots=f_p=0 \} \cap W$ and since $f_1$ is a carpeting function for $M_{F}$ the result follows.  \endproof

\vspace{0.2cm}

Let us denote by $M_F^\delta$ and $M_G^\delta$ the manifolds with boundary $\{f_1\ge \delta, f_2=\ldots=f_p=0 \} \cap W$ and $\{f_1 \ge \delta, f_2=\ldots=f_{p-1}=0 \} \cap W,$ respectively. Observe that $M_F^\delta= M_G^\delta \cap \{\bar{f}_p = 0 \}$ and remind that $\bar{f}_p$ denotes the function $\frac{f_p}{\Vert F \Vert}.$ We will apply Morse theory to the function $\bar{f}_p$ restricted to $M_G^\delta$.

\vspace{0.2cm}

For this, let $\psi=(f_2,\ldots,f_{p-1})$ and assume that $x$ is a critical point of $\bar{f_p} _{\vert W \cap V(\psi)},$ where $f_1(x) \not= 0.$
Then we have

$${\rm rank} \left[ \begin{array}{c}
\nabla h_1(x) \cr
\vdots \cr
\nabla h_{N-n+1} (x) \cr
\nabla f_2 (x) \cr
\vdots \cr
\nabla f_{p-1}(x) \cr
\nabla \bar{f}_p(x)  \cr
\end{array} \right] < N-n+1 + p-1.$$
On $V(\psi)$, we have $\nabla \bar{f}_i(x)= \frac{1}{\Vert F (x) \Vert} \nabla f_i(x)$, for $i \in \{2,\ldots,p-1 \}$. Hence
$${\rm rank} \left[ \begin{array}{c}
\nabla h_1(x) \cr
\vdots \cr
\nabla h_{N-n+1} (x) \cr
\nabla \bar{f}_2 (x) \cr
\vdots \cr
\nabla \bar{f}_{p-1}(x) \cr
\nabla \bar{f}_p(x)  \cr
\end{array} \right] < N-n+1 + p-1.$$
Since $f_1(x) \not= 0$, the above inequality shows that $x$ belongs to $\Sigma^W_{\bar{F}}$ by Remark \ref{re:generators}. But by Condition (b), $\Sigma^W_{\bar{F}}$ is empty. It follows that $\bar{f_p}_{\vert W \cap V(\psi)}$ has no critical point on $\{f_1 \not= 0\}$.

Our next step is to investigate the critical points of $\bar{f}_p$ restricted to
$$\partial M_G^\delta = \{f_1= \delta, f_2=\ldots=f_{p-1}=0 \} \cap W .$$

\vspace{0.2cm}

\begin{lemma}\label{lem:singularities1}
For $\delta$ sufficiently small, the critical points of $\bar{ f_p}_{\vert \partial M_G^\delta}$ are correct and lie in $\{f_p \not= 0 \}$.
\end{lemma}
\proof The critical points of $\bar{ f_p}_{\vert \partial M_G^\delta}$ are correct because $\bar{f_p}_{\vert W \cap V(\psi)}$ has no critical points in $\{ f_1 \not= 0 \}$. As $f_1$ is a carpeting function for $M_F$, we can choose $\delta$ sufficiently small in such a way that
$${\rm rank}
 \left[ \begin{array}{c}
\nabla h_1 (x) \cr
\vdots \cr
\nabla h_{N-n+1} (x) \cr
\nabla f_1 (x) \cr
\vdots \cr
\nabla f_{p-1}(x) \cr
\nabla f_p(x)  \cr
\end{array} \right] = N-n+1+p,$$
on $M_F \cap \{f_1 = \delta \}$. Note that on $M_F \cap \{f_1 = \delta \}$, the following equality holds:
$$\nabla \bar{f}_p = \delta^{-1} \nabla f_p,$$
since $\nabla \bar{f}_p = \frac{1}{\Vert F \Vert} \nabla f_p$. Replacing $\nabla f_p(x)$ by $\nabla \bar{f}_p$ in the matrix above, we conclude that
the critical points of $\bar{ f_p}_{\vert \partial M_G^\delta}$ lie in $\{f_p \not= 0 \}$.
\endproof

\begin{lemma}\label{lem:singularities2}
If $\delta$ is small enough then the critical points of $\bar{f_p}_{\vert \partial M_G^\delta}$ in $\{f_p \not= 0 \}$ are outwards-pointing (resp. inwards-pointing) for $M_G^\delta$ if and only if they lie in $\{ f_p >0 \}$ (resp. $\{ f_p < 0 \}$).
\end{lemma}
\proof Suppose the lemma is false. Then, for $\delta>0$ small enough we can find a critical point $x_\delta$ of $\bar{f_p}_{\vert \partial M_G^\delta}$ which is inwards-pointing and belongs to $\{ f_p > 0 \}$. Therefore, there exists $\lambda(x_\delta) >0$ such that
$$\nabla \bar{f_p}_{\vert V(\psi) \cap W} (x_\delta) = \lambda (x_\delta) \nabla {f_1}_{\vert V (\psi) \cap W} (x_\delta).$$

By the Curve Selection Lemma, we can find an analytic curve $\gamma(t)\subset V(\psi)\cap W$ defined on a small enough interval $[0,\nu[$ such that $f_1(\gamma(0))= 0$, and for $t \in ]0,\nu[$:
\begin{itemize}
\item $f_p(\gamma(t))>0$,
\item $f_1(\gamma(t)) >0$,
\item $\nabla \bar{f_p}_{\vert V(\psi) \cap W} (\gamma(t)) = \lambda (\gamma(t)) \nabla {f_1}_{\vert V (\psi) \cap W} (\gamma(t))$ with $\lambda(\gamma (t)) >0$.
\end{itemize}
Thus we get
$$(\bar{f_p} \circ \gamma)'(t)= \langle \nabla \bar{f_p}_{\vert V(\psi) \cap W} (\gamma(t)),\gamma'(t) \rangle=$$ $$\lambda(\gamma(t)) \langle \nabla {f_1}_{\vert V(\psi) \cap W} (\gamma(t)),\gamma'(t) \rangle=
\lambda(\gamma (t))(f_1 \circ \gamma)'(t).$$
Assume that $(f_1 \circ \gamma)' \le 0$ and therefore the function $f_1 \circ \gamma $ decreases. Since $f_1(\gamma(t))$ tends to $0$ as $t$ tends to $0$, it follows that $f_1 \circ \gamma (t) \le 0$, which is impossible. Hence $(f_1 \circ \gamma)'>0$. We deduce that $(\bar{f_p} \circ \gamma)'>0$ and $\bar{f_p} \circ \gamma$ is strictly increasing.

Let us evaluate $f_p (\gamma(0))$. If $f_p(\gamma(0))=0$ then $\gamma(0) \in V_W(F)$. Furthermore, we know that
$$\nabla \bar{f_p}(\gamma(t))= \frac{1}{\Vert F(\gamma(t)) \Vert^3}  f_1^2 (\gamma(t)) \nabla f_p (\gamma(t)) -f_p(\gamma(t))f_1(\gamma(t)) \nabla f_1(\gamma(t)).$$
This implies that $\gamma(t)$ belongs $\Sigma_F^W\setminus V_W(F)$ for $t>0,$ since $\gamma(t)$ is a critical point of $\bar{f_p}_{\vert \partial M_G^{\delta(t)}},$ where $\delta(t)=f_{1}(\gamma(t))$. Therefore, $\gamma(0)$  belongs to $\overline{\Sigma_F^W\setminus V_W(F)} \cap V_W(F)$ which is empty. Consequently, $f_p(\gamma(0)) >0$ and $\bar{f_p}(\gamma(0))= 1$ because $f_1(\gamma(0))=\ldots=f_{p-1}(\gamma(0))=0.$

Since $\bar{f_p} \circ \gamma (t)$ tends to $1$ when $t$ tends to $0$, then the monotonicity of $\bar{f_p} \circ \gamma$ gives that $\bar{f_p}(\gamma(t))>1$, which contradicts the fact that $\vert \bar{f_p} \vert \le 1$. \endproof

Now we have the main ingredients to prove Theorem \ref{t:homotopy}.

\proof (Theorem \ref{t:homotopy})
Applying the previous lemmas and Morse theory for manifolds with boundary, we get that $M_G^\delta \cap \{f_p \ge 0\}$ is homeomorphic to $M_F^{\delta}\times [0,1]$ and $M_G^\delta \cap \{f_p \le 0\}$ is homeomorphic to $M_F^{\delta}\times [-1,0]$. Therefore $M_G^\delta $ is homeomorphic to $M_F^\delta \times [-1,1]$ which is a consequence of the gluing $M_G^\delta \cap \{f_p \ge 0\}$ and $M_G^\delta \cap \{f_p \le 0\}$ along
$M_G^\delta \cap \{f_p=0 \}= M_F^\delta$. Now, we can use Corollary \ref{HomEqu} to finish the proof. \endproof

\subsection{\bf{Case $p=2$}}

Let $F=(f_1,f_2): \mathbb{R}^N \rightarrow \mathbb{R}^2$ be a $C^2$ semi-algebraic map satisfying Conditions (a) and (b).
Let us denote by $M_{f_1}^+$ the set $\{f_1 >0 \} \cap W$ and by $M_{f_1}^-$ the set $\{f_1 < 0 \} \cap W$

\begin{theorem}\label{t:Case2}
The fiber $M_F$ is homotopy equivalent to $M_{f_1}^+$ and $M_{f_1}^-$.
\end{theorem}
\proof The proof is the same as Theorem \ref{t:homotopy}. We just have to check that $f_1$ is a carpeting function for $M_{f_1}^+$ and $-f_1$ is a carpeting function for $M_{f_1}^-$. Since ${f_1}_{\vert W}: W \rightarrow \mathbb{R}$ has a finite number of critical values, then there exists $\delta >0$ such that $]0,\delta]$ and $[-\delta,0[$ do not contain any critical value of ${f_1}_{\vert W}$. \endproof

\subsection{\textbf{Consequences}}
Let $F=(f_1,\ldots,f_p): \mathbb{R}^N \rightarrow \mathbb{R}^p$, $2 \le p \le n-1$, be a $C^2$ semi-algebraic map satisfying Conditions (a) and (b). Let us consider $l \in \{1,\ldots,p \}$ and $I= \{i_1,\ldots,i_l\}$ an $l$-tuple of pairwise distinct elements of $\{1,\ldots,p\}$ and let $f_I$ be the semi-algebraic map $(f_{i_1},\ldots,f_{i_l}): \mathbb{R}^N \rightarrow \mathbb{R}^l$. We know that if $l \ge 2$, then the map $f_I$ satisfies Conditions (a) and (b). Let us also denote by $M_{f_I}$  the fiber of $\bar{f_I}_{\vert W\setminus V_{W}(f_{I})}.$
\begin{corollary}
Let $l \in \{2,\ldots,p\}$ and let $I=\{i_1,\ldots,i_l \}$ be an $l$-tuple of pairwise distinct elements of $\{1,\ldots,p\}$. Then, the fibers $M_{f_I}$ and $M_F$ are homotopy equivalent.
\end{corollary}
\proof Apply Theorem \ref{t:homotopy}. \endproof

Let us consider the case $l=1$. From the facts explained above we can state the following.
\begin{theorem}\label{t:homotopy2}
For every $j \in \{1,\ldots,p\}$, the fibers $M_{f_{j}}^+$ and $M_{f_{j}}^-$ are homotopy equivalent to the fiber $M_F$.
\end{theorem}
\proof Apply Theorem \ref{t:Case2}. \endproof

\section{Euler characteristics of fibers and relative links }\label{section5}
In this section, we give several formulae connecting the Euler characteristics of the sets $V_W(f_I),$ which we call the relative links (or links, for short), and the Euler characteristic of the fiber $M_F$. Let us follow the same notations and assumptions for $F$ as in the previous sections.

\vspace{0.2cm}

We prove first a general result which relates the Euler characteristics of the fiber and the link and, as a consequence, we get an obstruction for the links to be homotopy equivalent after projections.

\begin{proposition}\label{p:Euler} We have:
$$\chi \big(V_W(G) \big) - \chi \big( V_W(F)  \big)= (-1)^{n-p} 2.\chi (M_F).$$
\end{proposition}
\proof
Let us denote by
$$ M_F^+ = \left\{ f_1 =f_2=\ldots=f_{p-1}=0, f_p >0 \right\} \cap W = V_W(G) \cap \{f_p >0 \},$$
and
$$ M_F^- = \left\{ f_1 =f_2=\ldots=f_{p-1}=0, f_p <0 \right\} \cap W = V_W(G) \cap \{f_p <0 \}.$$
Observe that these two sets are diffeomorphic to the fiber $M_{F}$.
We know that $f_p$ is a carpeting function for $M_F^+$ and $-f_p$ is a carpeting function for $M_F^-$. Therefore, there exists $0<\delta \ll 1$ such that $M_F^+$ is homotopy equivalent to $V_W(G)  \cap \{f_p \ge \delta \}$ and $M_F^-$ is homotopy equivalent to $V_W(G) \cap \{f_p \le -\delta \}$.

By the Mayer-Vietoris sequence we can write
$$\chi(V_W(G)) = \chi (V_W(G) \cap \{ f_p \ge \delta \} )+\chi (V_W(G) \cap \{ f_p \le -\delta \} )+ $$
$$\chi (V_W(G) \cap  \{ -\delta \le f_p \le \delta \} ) -\chi(V_W(G) \cap  \{ f_p = \delta \} )-$$ $$\chi(V_W(G) \cap \{ f_p = -\delta \} ).$$
By the above arguments, the first two terms on the right-hand side are equal to $\chi(M_F)$.
The third term is equal to $\chi( V_W(F))$, since by Durfee's result \cite{D} we can choose $\delta>0$ small enough in such a way that $V_W(F)$ is a retract by deformation of $V_W(G) \cap \{ -\delta \le f_p \le \delta \}$.

Furthermore, if $n-p$ is even then the two last Euler characteristics are equal to $0$ because $V_W(G) \cap  \{ f_p = \delta \} $ and $V_W(G) \cap  \{ f_p =- \delta \} $ are odd-dimensional closed manifolds. If $n-p$ is odd, they are equal to $2\chi(M_F)$ because they are the boundaries of the odd-dimensional manifolds $M_F^\delta=V_W(G) \cap \{f_p \geq \delta \},$ Therefore, the result is proved.
\endproof

Let us choose $l \in \{2,\ldots,p\}$ and an $l$-tuple $I=\{i_1,\ldots,i_l\}$ of pairwise distinct elements of $\{1,\ldots,p\}$. Let $J=\{i_1,\ldots,i_{l-1} \}$.
If $l=1$ then we put $J=\emptyset$ and $f_J=0$. We can now formulate the previous proposition in a more general way as follows.
\begin{proposition}\label{LINK} We have:
$$\chi (V_W(f_J))-\chi (V_W(f_I))=(-1)^{n-l} 2 \chi (M_F).$$
\end{proposition}
\proof Same proof as Proposition \ref{p:Euler}. \endproof
In particular, we obtain the following two corollaries
\begin{corollary}
Let $j \in \{1,\ldots,p \}$, then we have $$\chi(V_W(f_j))=\chi(W)-(-1)^{n-1} 2\chi(M_F).$$
\end{corollary}
\proof We apply Proposition \ref{LINK} to the case $l=1$.
\endproof

\begin{corollary}
Let $l\in \{3,\ldots,p \}$ and let $I=\{i_1,\ldots,i_l\} \subset \{1,\ldots,p\}$. Let $K$ be an $(l-2)$-tuple of pairwise distinct elements of $I$. Then we have: $$\chi (V_W(f_K))=\chi (V_W(f_I)).$$
\end{corollary}
\proof Let $J$ be an $(l-1)$-tuple built from adding to $K$ one element of $I \setminus K$. By Proposition \ref{LINK}, the corollary follows from equality $$\chi (V_W(f_J))-\chi (V_W(f_I))=\chi (V_W(f_J))-\chi (V_W(f_K)).$$
\endproof

In order to express the Euler characteristics of all the sets $V_W(f_I)$, it is sufficient to compute the Euler characteristic of a set $V_W(f_I)$ when $\# I=2$. But, for $I=\{1,2\}$ it follows by Proposition \ref{LINK} that
$$\chi(V_W(f_I))=\chi(V_W(f_1))-(-1)^n 2\chi(M_F).$$
Therefore, we can summarize all these results in the following theorem.

\begin{theorem}\label{t:main}
Let $l\in \{1,\ldots,p \}$ and let $I=\{i_1,\ldots,i_l\}$ be an $l$-tuple of pairwise distinct elements of $\{1,\ldots,p\}$.

\begin{itemize}

\item [(1)] For $n$ even, we have  $$\chi(V_W(f_I))= 2\chi(M_F) \hbox{ if } l \hbox{ is odd} ,$$
 and $$\chi(V_W(f_I))= 0 \hbox{ if } l\hbox{ is even}.$$

\vspace{0.2cm}

\item [(2)] For $n$ odd, we have  $$\chi(V_W(f_I))=\chi(W)- 2\chi(M_F) \hbox{ if } l \hbox{ is odd},$$ and $$\chi(V_W(f_I))=\chi(W)\hbox{ if } l \hbox{ is even}.$$

\end{itemize}

\end{theorem}

\section{Open book structure of polynomial maps and examples}
In this section, we consider our Conditions (a) and (b) for polynomial maps in the local and global settings.
Our purpose is to explain the so-called ``Higher open book structure with singular binding" induced by $\bar{F}$. For further details we refer the reader to \cite{AT,ACT1,ACT2}.

\vspace{0.2cm}

\paragraph{\textbf{Local setting}}
Let $f=(f_1,\ldots,f_p): (\mathbb{R}^N,0) \rightarrow (\mathbb{R}^p,0)$ be a real analytic map germ.  In \cite[Theorem 11.2]{Mi} Milnor proved that if $f$ has \emph{isolated singularity} at origin, then there exists a locally trivial fibration $S^{N-1}_\epsilon\setminus K_\epsilon \rightarrow S^{p-1}$. Although $\bar{f}$ may fail to provide the fibration structure outside a small neighborhood of the link, in \cite{AT, ACT1} the authors consider two similar conditions like (a) and (b) above, in order to prove a generalization of Milnor's fibration on small spheres $S^{N-1}_\epsilon$ with   bundle projection given by $\bar{f}.$

\vspace{0.2cm}

Applying our results in this case, we consider $\rho(x)=\sum_{i=1}^{N}x_{i}^{2}$ the square of distance function to the origin and for $\epsilon>0$ small enough consider $W=\rho^{-1}(\epsilon ^{2}).$ So, $V_W(f)=K_\epsilon$ is the link of the germ $V(f)$. The singular locus $\Sigma_f^W$ (resp. $\Sigma_{\bar{f}}^W$) is then the intersection between the Milnor sets $M(f)=\{ x\in U \mid f \not\pitchfork_x \rho \}$ (resp. $M(\bar{f})=\{ x\in U\setminus V(f) \mid \bar{f} \not\pitchfork_x \rho \})$ and $S_\epsilon^{N-1}$. Therefore, Conditions (a) and (b) above coincide with those conditions studied in \cite{ACT1} and so the Theorem 1.3 in \cite{ACT1} is a special version of our Theorem \ref{t:fibration}. See \cite{ACT1,ACT2,NZ} for further details.

\begin{example}\label{e:Milnor}
Let $f:(\mathbb{R}^{3},0)\rightarrow (\mathbb{R}^{2},0)$, $f(x,y,z)=(x, x^2+y(x^2+y^2)+z^2)$. Then $f$ has only an isolated singularity $(0,0,0)$ and therefore Condition (a) holds. But, for $\epsilon>0$ small enough the projection $\frac{f}{\|f\|}:S^{2}_\epsilon\setminus K_\epsilon\rightarrow S^{1}$ cannot be a locally trivial fibration. Calculations shows that $\{(x,y,0)\in\mathbb{R}^{3}\setminus V(f)| x^4+y^4+2x^2y^2-xy^2=0\}\subset M(\bar{f})$ is a bounded curve such that the origin is contained in its closure, therefore locally Condition (b) fails.
\end{example}

\subsubsection{Homotopy equivalence between local Milnor's fibers}Here we present an application of Theorem \ref{t:homotopy2}, that gives a relation between the Milnor
fiber in the tube and the Milnor fiber in the sphere.
Let $f=(f_1,\ldots,f_p): (\mathbb{R}^n,0) \rightarrow (\mathbb{R}^p,0)$ be a real analytic map germ satisfying Massey's Conditions (a) and (b). Then, it was proved by Massey in \cite{DM} that there exists a Milnor fibration in the tube

$$f_{|}:B_{\epsilon}\cap f^{-1}(S_{\eta}^{p-1})\to S_{\eta}^{p-1}$$ for all small enough $\epsilon >0$ and $0<\eta \ll \epsilon.$ Assume also that the Milnor set $M(\frac{f}{\|f\|})$ is empty. So, by \cite{ACT1} as we explained above, the projection 

$$\frac{f}{\|f\|}: S_{\epsilon}^{n-1}\setminus K_{\epsilon}\to S^{p-1}$$
 is a locally trivial fibration.
 
\begin{proposition} 
The two fibers above $M_{f}$ and $M_{\frac{f}{\|f\|}}$ are homotopy equivalent. 
\end{proposition}

\proof By \cite{DA} we know that $M_{f}$ is homotopy equivalent to $M_{f_{1}^{+}},$ i.e., the positive Milnor fiber of $f_{1}.$ Then, by Theorem \ref{t:homotopy2} the fiber $M_{\frac{f}{\|f\|}}$ is homotopy equivalent to $S_{\epsilon}^{n-1}\cap \{f_{1}>0\}.$ But by C.T.C. Wall \cite{W}, $M_{f_{1}^{+}}$ and $S_{\epsilon}^{n-1}\cap \{f_{1}>0\}$ are homotopy equivalent.  \endproof

\paragraph{\textbf{Global setting}}
Now consider $F=(f_1,\ldots,f_p): \mathbb{R}^N \rightarrow \mathbb{R}^p$ a real polynomial mapping. In \cite{ACT2} the authors studied the singular open book structure at infinity by using the $\rho$-regularity. In fact, if we choose $W=\rho^{-1}(R)$ with $R$ big enough, our Conditions (a) and (b) means the boundedness of the sets $\overline{M(F)\setminus V}\cap V$ and $M(\bar{F}),$ where $V=F^{-1}(0).$ Therefore, by our Theorem \ref{t:fibration} we recover Theorem 3.2 of \cite{ACT2}.

\vspace{0.2cm}

Now we can consider the example \ref{e:Milnor} above in the global setting. We will show that for $R$ sufficiently large, $\frac{f}{\|f\|}:S^{2}_R\setminus K_R\rightarrow S^{1}$ is a locally trivial fibration.

In fact, since it has only one singular point, then Condition (a) is satisfied.

Now, the set $M(\bar{f})\cap \{z=0\}$ is equal to $\{(x,y,0)\in\mathbb{R}^{3}\setminus V(f)| x^4+y^4+2x^2y^2-xy^2=0\}$ which is a bounded set. The set $M(\bar{f})\cap \{z\neq 0\}$ is equal to $\{(x,y,z)\in\mathbb{R}^{3}\setminus V(f)| 2y=x^2+3y^2, x^2+y^3+z^2-x^2y=0\}$ which is also a bounded set because $0<y<\frac{2}{3}$, $x^2\leq\frac{1}{3}$ and $z^2<\frac{2}{9}.$ So, in both cases our Condition (b) holds for $f$. Therefore, by Theorem \ref{t:fibration} we obtain the desired fibration.


We end  this section with two applications of Theorem \ref{t:main} in the global setting. Note that in this case $V_W(f_I)$ is the link at infinity of $\{f_{I}=0\},$ because we consider $W$ as a sphere of radius big enough. It is known that it is well defined up to homeomorphism. For this reason we denote it by ${\rm Lk}^\infty (V_I)$ instead of $V_W(f_I).$
\begin{example}\label{e:polar}
\emph{Let $f:\mathbb{C}^{2}\rightarrow\mathbb{C}$, $f(x,y)=x(x+y)\bar x$. From \cite{ACT2}, this example is polar weighted homogenous mixed polynomial which verifies Conditions (a) and (b) at infinity. Therefore the fibration $\bar{f}$ exists. The fiber is homotopy equivalent to $\mathbb{R}^2$ and the Euler characteristic of the fiber is $1$. On the other hand, we can write $f$ as a polynomial map $f:\mathbb{R}^{4}\rightarrow\mathbb{R}^2$, $f(a,b,c,d)=((a^2+b^2)(a+c),(a^2+b^2)(b+d))$. Hence the Euler characteristic of the link $\chi({\rm Lk}^\infty ((a^2+b^2)(a+c)=0))=2$, by using our Theorem \ref{t:main}, we get the Euler characteristic of the fiber is $1$.}
\end{example}
\begin{example}\label{e:smooth}
\emph{Let $f:\mathbb{R}^{3}\rightarrow\mathbb{R}^{2}$, $f(x,y,z)=(x^2+y, x+z)$. Then $f$ does not have any singularity. Moreover by computation, the Milnor set $M(\bar{f})$ is empty. Therefore $f$ verifies Conditions (a) and (b) at infinity and the fibration $\bar{f}$ exists. The fiber is homeomorphic to an arc of a circle which has Euler characteristic equal to $1$. On the other hand $\chi({\rm Lk}^\infty (x+z=0))=0$, by using our Theorem \ref{t:main}, we get that the Euler characteristic of the fiber is $1$.}
\end{example}

{\it Acknowledgments.} The first two authors thank the USP-Cofecub project ``UcMa133/12 - Structure fibr\'ee de l'espace au voisinage des singularit\'es des applications". The author Y. Chen is supported by the Brazilian grant FAPESP-Proc. 2012/18957-7. The author A. Andrade is supported by Capes pro-doutoral grant.

\end{document}